\documentclass[a4paper,11pt]{amsart}

\usepackage{amsmath,amssymb,amsthm,a4wide}    
\usepackage{enumerate}

\newtheorem{thm}{\bf Theorem}[section]
\newtheorem{lem}[thm]{\bf Lemma}
\newtheorem{pro}[thm]{\bf Proposition}
\newtheorem{coro}[thm]{\bf Corollary}

\newtheorem{rmk}{\bf Remark}

\theoremstyle{definition}
\newtheorem{definition}[thm]{Definition}

\newcommand{\Ric}{{\text{Ric}}}
\newcommand{\RC}{{\text{RC}}}
\newcommand{\RCD}{{\text{RCD}}}
\newcommand{\CD}{{\text{CD}}}
\newcommand{\MCP}{{\text{MCP}}}
\newcommand{\Lip}{{\text{Lip}}}
\newcommand{\Alex}{{\text{\bf Alex}}}
\newcommand{\BG}{{\text{BG}}}
\newcommand{\Haus}{{\mathcal{H}}}
\newcommand{\gexp}{{\text{gexp}}}
\newcommand{\cut}{{\text{cut}}}
\newcommand{\Cbar}{{\overline{C}}}

\newcommand{\vol}{{\text{vol}}}
\newcommand{\ka}{{\kappa}}
\newcommand{\ska}{{s_{\kappa}}}
\newcommand{\loc}{\text{loc}}
\newcommand{\vecpq}{{\stackrel{\rightarrow}{pq}}}

\numberwithin{equation}{section}

\begin{document}

\title[Alexandrov Spaces with Large Volume Growth]{Alexandrov Spaces with Large Volume Growth}

\author{Michael Munn}
\thanks{Department of Mathematics, University of Missouri, 210 Math Sciences Building, Columbia, MO ~65211. Email: \texttt{munnm@missouri.edu}.\\}

\maketitle

\begin{abstract}
Let $(X,d)$ be an $n$-dimensional Alexandrov space whose Hausdorff measure $\Haus^n$ satisfies a condition giving the metric measure space $(X,d,\Haus^n)$ a notion of having nonnegative Ricci curvature. We examine the influence of large volume growth on these spaces and generalize some classical arguments from Riemannian geometry showing that when the volume growth is sufficiently large, then $(X,d,\Haus^n)$ has finite topological type. 
\end{abstract}

\section{Introduction}
\label{section-introduction}
In this paper, we study $n$-dimensional Alexandrov spaces $(X,d)$ whose Hausdorff measure $\Haus^n$ satisfies a condition giving the metric measure space $(X,d,\Haus^n)$ a notion of having nonnegative Ricci curvature. In particular, we examine the influence of large volume growth (as defined using the Hausdorff measure) on these spaces and generalize some classical results from Riemannian geometry. 

Alexandrov spaces arise as the Gromov-Hausdorff limits of $n$-dimensional, compact Riemannian manifolds with sectional curvature $\geq k$ and diameter $\leq D$.  By now the study of Alexandrov spaces has become an interesting subject on its own \cite{BGP, PerelmanII, BBI, Shiohama}. In addition, in light of Gromov's Precompactness Theorem  \cite{Gromov} and Perelman's Stability Theorem (see Theorem \ref{PerelmanStability}), understanding the topology of Alexandrov spaces is also important for studying the curvature and topology of Riemannian manifolds. Already many well-known results from Riemannian geometry have been generalized to the more broad class of Alexandrov spaces (e.g. \cite{HarveySearle, HuaI, HuaII, JiaoI, JiaoII, KohI, KohII, KuwaeShioyaI, KuwaeShioyaII} among others) and many of the necessary tools used for smooth Riemannian manifolds have been adapted to these general metric spaces as well. 

While Alexandrov spaces capture the notion of lower sectional bounds for metric spaces, there are also several notions which aim to generalize the notion lower Ricci curvature bounds to metric measure spaces. In this direction, Sturm \cite{SturmI, SturmII} and Lott-Villani \cite{LottVillani} independently introduced a curvature-dimension condition $\CD(K, n)$, with $n \in (1, \infty]$ and $K \in \mathbb{R}$, defining metric measure spaces with $\Ric \geq (n-1)K$ and dimension $\leq n$. Their definition requires the convexity of specific entropy functionals when measured along geodesics in the $L^2$-Wasserstein space (see also \cite{BacherSturm}). Related to their definitions, Ohta \cite{Ohta} introduced the measure contraction property $\MCP(K,n)$ which is closely related to the definitions of Sturm-Lott-Villani. In short, the condition $\MCP(K,n)$ is an infinitesimal version of the Bishop-Gromov volume comparison theorem which is a well-known consequence of $n$-dimensional Riemannian manifolds with Ricci curvature $\geq (n-1)K$. Very recently, Ambrosio-Gigli-Savare \cite{AGSII} added the additional assumption of infinitesimally Hilbertian to the $\CD$ condition of Sturm-Lott-Villani which also ensures linearity of the Laplacian on the metric measure space to define the Ricci curvature dimension condition, denoted $\RCD(K, n)$. 

In this paper, we are concerned with Alexandrov spaces. In this setting, Zhang-Zhou \cite{ZhangZhouII} have introduced a new definition of lower bounds of Ricci curvature, denoted $\RC$, which is based on the the behavior of the second variation formula of arc length known for Riemannian manifolds. Also, in the setting of Alexandrov spaces, is the Bishop-Gromov  condition $\BG(K,n)$ introduced by Kuwae-Shioya. This definition can be seen as a special case of Ohta's MCP condition applied to Dirac masses and the Hausdorff measure on $(X,d)$. 

Clearly, viewing $(X,d,\Haus^n)$ as a metric measure space, each of the previous definitions described above also make sense for Alexandrov spaces. In fact, by work of Burago-Gromov-Perelman \cite{BGP},  Alexandrov spaces are necessarily infinitesimally Hilbertian. So, for an $n$-dimensional Alexandrov space with Hausdorff measure $\Haus^n$ we have:

\vspace{-.2in} 

{\small
\hspace{-.1in} \[\RC \!\geq\! (n-1)K\!\! \implies\!\! \RCD((n-1)K,n)\!\! \iff  \!\!\CD((n-1)K,n)\!\! \implies\!\! \MCP((n-1)K,n)\!\! \iff \!\!\BG((n-1)K, n)\]
}
\vspace{-.2in} 

In the same way that classical results for Riemannian manifolds with lower sectional curvature bounds have been generalized to Alexandrov spaces, there has been a strong effort to generalize the well-known theorems of the comparison geometry of Ricci curvature to metric measure spaces satisfying some weak Ricci curvature bound condition described above (e.g. see \cite{KuwaeShioyaI, KuwaeShioyaII, GigliMosconi, ZhangZhouI} for various generalizations of the well-known Cheeger-Gromoll Splitting Theroem.) To a large extent, this paper contributes to that goal. We consider Alexandrov spaces whose measure satisfies the $\BG(0,n)$ condition (see Section \ref{subsection-BG} for complete definition) which is the weakest of all the notions described above. We prove 

\begin{thm}
\label{mainthm}
For an integer $n \geq 2$, let $(X,d)$ be an $n$-dimensional, complete, noncompact Alexandrov space with nonempty boundary and whose Hausdorff measure satisfies the $BG(0,n)$ condition. There exists an $\epsilon(n,\ka) >0$ such that if 
\[ \Haus^n(B(p,r)) \geq  (1-\epsilon)\omega_n r^n, \quad \text{ for all } r >0, \text{ and some }p \in X,\]
then $(X,d)$ has finite topological type. 
\end{thm}

When phrased in this way we assume that $(X,d)$ has some arbitrary lower bound on its Alexandrov curvature (see Section \ref{subsection-alexandrov} for complete definition). With this in mind, Theorem \ref{mainthm} provides an ideal metric geometry generalization of a classical type of theorem for Riemannian manifolds with a lower  Ricci curvature bound and whose sectional curvature $\geq -\ka^2$ (e.g. \cite{AbreschGromoll1990, doCarmoXia, Eschenburg, Ordwayetal, PerelmanDiameterStability, Shen, Xia} among others). We expect that each of these results would also admit similar generalization as well. 

The proof of our main Theorem relies on an application of the excess estimate for Alexandrov spaces which we prove in this paper. Originally, the excess estimate was  proven by Abresch-Gromoll \cite{AbreschGromoll1990} for Riemannian manifolds with only a lower  Ricci curvature bound (i.e. no bound on sectional curvature). Here, we use a generalized version of the excess estimate for Alexandrov spaces whose Hausdorff measure satisfies an infinitesimal Bishop-Gromov condition so that we say $(X,d, \Haus^n)$ has nonnegative Ricci curvature. Our method of proof follows the original proof of Abresch-Gromoll \cite{AbreschGromoll1990}. 

In fact, an even more general version of the excess estimate was recently generalized by Gigli-Mosconi \cite{GigliMosconi} to metric measure spaces satisfying a curvature dimension condition $\CD(K,n)$ and which are infinitesimally Hilbertian. Note that this version of Gigli-Mosconi holds for any infinitesimally Hilbertian metric measure space (not necessarily Alexandrov) and thus contains ours. We include a sketch of the proof of the excess estimate in the Appendix of this paper and highlight the main ideas of the argument there.

In Section \ref{section-preliminaries} we define Alexandrov spaces and mention some of their basic analytic features which will be useful in the sequel. We also define the Bishop-Gromov condition of Kuwae-Shioya which we mentioned briefly above. In Section \ref{section-proofs}, we prove Theorem \ref{mainthm}. Our proof relies on a typical application of the excess estimate which we verify holds for Alexandrov spaces with the $\BG(0,n)$ condition in the Appendix.
\section{Preliminaries}
\label{section-preliminaries}
\subsection{Alexandrov Spaces} 
\label{subsection-alexandrov}In this subsection, we recall some of the basic information of Alexandrov spaces (see also \cite{BBI, OtsuShioya, BGP, PerelmanII, PerelmanPetrunin}, in particular equivalent definitions in Chapter 4 of \cite{BBI}).

A finite-dimensional {\it Alexandrov space} is a complete, locally compact, connected length space which satisfies a lower curvature bound in an angle-comparison sense described below. Recall that a length space (or geodesic space) is a metric space $(X,d)$ such that any two points $p,q \in X$ can be joined by a rectifiable curve whose length is equal to $d(p,q)$. We call such a distance realizing curve a {\it minimal geodesic} and use $\overline{pq}$ to denote a (not necessarily unique) minimal geodesic joining $p$ and $q$. 

To describe the lower curvature bound, fix a number $k \in \mathbb{R}$. For any triple of points $p,q,r \in X$ (usually thought of as a triangle $\Delta pqr$), let $\tilde{\Delta}pqr$ denote the triangle $\Delta \tilde{p}\tilde{q}\tilde{r}$ of points in the  $k$-plane such that $|\tilde{p}\tilde{q}| = d(p,q), |\tilde{q}\tilde{r}| = d(q,r)$ and $|\tilde{p}\tilde{r}| = d(p,r)$. 
\begin{definition}
$(X,d)$ has curvature $\geq k$ in an open set $U \subset X$ if, for each quadruple of points $(p; a,b,c)$ in $U$,  
\[\widetilde{\angle}_{k}apb + \widetilde{\angle}_{k}bpc +\widetilde{\angle}_{k}cpa \leq 2\pi,\]
where $\widetilde{\angle}_{k}apb$ is the comparison angle at $\tilde{p}$ of a triangle $\tilde{\Delta}apc$ in the $k$-plane (define $\widetilde{\angle}_{k}bpc$, $\widetilde{\angle}_{k}cpa $ similarly).
\end{definition}
In \cite{BGP} (c.f. \cite{BBI}), the authors show that in fact this local condition implies the same is true globally for any quadruple of points in $X$. 

We will denote the class of $n$-dimensional Alexandrov spaces of curvature $\geq k$ by $\Alex^n[k]$. Certain care must be taken when assuming $k >0$ to ensure these quantities are well-defined; however, in this paper we focus on the case when $k$ is some arbitrary negative lower bound $k = -\ka^2 > -\infty$. Also, we assume that $X$ has no boundary. Note that by property of being Alexandrov, it follows that $(X,d)$ is a non-branching metric space.  Thus, using the terminology of Ambrosio-Gigli-Savare \cite{AGSII}, and Alexandrov space is infinitesimally Hilbertian thus  $\RCD(K, \infty) \iff \CD(K, \infty)$ as we saw above.

For $p \in X$, denote by $\Sigma_pX$ the {\em space of directions at $p$}. Note that $\Sigma_pX \in \Alex^{n-1}[1]$ and is compact. The metric cone over $\Sigma_pX$ equipped with the cone metric is called the {\em tangent cone at $p$} and is denoted $C_pX$; furthermore, $C_pX \in \Alex^n[0]$. Note, for smooth Riemannian manifold $(M,g)$ viewed as a metric space $(M, d_g)$ whose distance metric $d_g$ is induced from $g$, that $C_pM$ and $\Sigma_pM$ are precisely the tangent space $T_pM$ and the unit tangent sphere $S_pM$ (resp.) of $M$ at $p$.

\begin{definition}{\em \bf (Singularities of Alexandrov Spaces.)}
A point $p$ is called a {\it singular point of X} if $\Sigma_pX$ is not isometric to the unit sphere $\mathbb{S}^{n-1}$. For some $\delta>0$, a point is called  {\it $\delta$-singular} if $\Haus^{n-1}(\Sigma_pX) \leq \vol(\mathbb{S}^{n-1}) - \delta$. Let $S_X$ and $S_{\delta}$ denote the set of singular points of $X$ and the set of $\delta$-singular points of $X$ (respectively). Note $S_X = \bigcup_{\delta >0} S_{\delta}$.
We say that a point $p$ is {\it topologically singular} if  $\Sigma_pX$ is not homeomorphic to a sphere. 
\end{definition} 

Every finite-dimensional Alexandrov space can be stratified into topological manifolds. In particular, if an Alexandrov space has no boundary, then the codimension of the set of topologically singular points is a least three. 

We state now some additional properties of the local structure of $X$ that will be useful in the sequel. Most notably, we have
\begin{thm}
\label{PerelmanStability} {\em (Perelman's Stability Theorem, c.f. \cite{Kapovitch})}
Let $X \in \Alex^n[k]$ be compact. Then there exists some $\epsilon := \epsilon(X)>0$ such that any $Y \in \Alex^n[k]$ with $d_{GH}(X,Y) < \epsilon$ is homeomorphic to $X$. Here $d_{GH}$ denotes the Gromov-Hausdorff distance on metric spaces.
\end{thm}

It follows from this that every point $p$ of a finite-dimensional Alexandrov space has a small metric neighborhood which is pointed homeomorphic to the tangent cone $C_pX$. In \cite{OtsuShioya}, Otsu-Shioya prove a number of important structural results describing the Riemannian structure of Alexandrov spaces. In particular, they show that for $(X,d)$ as above there exists a $\delta_n>0$ such that there is a canonical $C^{\infty}$ structure on $X \setminus S_{\delta_n}$ which is compatible with the the $DC$-structure constructed by Perelman in \cite{PerelmanDC}. 

In addition, they construct a unique $C^0$-Riemannian metric $g$ on $X \setminus S_X$ such that the distance function induced from $g$ coincides with the original distance metric $d$ on $X$. This metric $g$ is of locally bounded variation and the volume measure on $X \setminus S_{\delta_n}$ induced from $g$ is 
\[d\vol_g(x) = \sqrt{|g|}~d\Haus^n(x).\]

It should be noted that the non-singular set $X \setminus S_X$ is not a manifold as it is possible that the set $S_X$ of singular points can be quite `bad', (e.g. Example 2 in \cite{OtsuShioya} where $S_X$ is dense in $X$). However, if the size of the singular set is controlled in some way (e.g. if $\overline{S_X} \neq X$) then many of the classical techniques of smooth Riemannian manifolds admit direct generalizations to Alexandrov spaces. Note that Petrunin has shown 
 \begin{lem}\em{(c.f. \cite{Petrunin-ParallelTransportation}.)} Let $p, q \in X$ and let $\overline{pq}$ be a minimal geodesic joining them. Then for all points $x \in \overline{pq}\setminus \{p,q\}$ the space of directions $\Sigma_xX$ are isometric.
\end{lem}
Which, in particular, guarantees that any minimal geodesic joining two non-singular points $p, q \notin S_X$ is also non-singular for any $x \in \overline{pq}$. 

\subsubsection{Gradient Curves of Alexandrov Spaces {\em (c.f. \cite{PerelmanPetrunin}) }} 
\label{subsection-gradientcurves}
In the smooth setting of Riemannian manifolds, the Inverse Function Theorem guarantees that at any point $p \in M$, the exponential map $\exp_p : T_pM \to M$ is a diffeomorphism in some neighborhood of the origin and, furthermore, if the sectional curvature of $M$ is $\geq k$ then $\exp_p$ is a contraction map. 

For Alexandrov spaces, however, the situation is a bit more subtle. The domain of the exponential map on $C_pX$ does not in general contain an open neighborhood of the origin and thus may not be defined. Instead, one defines the gradient exponential map which are constructed as gradient curves of distance functions in $X$. This construction is described in detail in Section 3 of \cite{PerelmanPetrunin} where, for any $p\in X$ and $\xi\in \Sigma_pX$, the authors construct a unique complete dist$_p$-gradient curve $\gamma :[0,a) \to X$ with initial velocity vector $\gamma^{+}(0)=\xi$.

A locally Lipshitz function $f: U \subset X \to \mathbb{R}$ is called {\em semiconcave} if for every $p \in X$ there exits $\lambda \in \mathbb{R}$ and a neighborhood $U_x \ni x $ such that $f|_{U_x}$ is $\lambda$-concave; i.e. for any unit speed geodesic $\gamma \in U_x$, $f \circ \gamma(t) - \frac{\lambda}{2}t^2$ is concave.

A semiconcave function is differentiable at almost every point in $X$ and at points $q\in X$ where $f$ is differentiable, we say a vector $v \in C_qX$ is the {\em gradient of $f$ at $q$} provided $df(v) = |v|^2$ and the function $df(u)/|u|^2$ on $C_pX$ achieves its maximum for $u = v$. We use the suggestive notation $\nabla d_p(q)$ to denote $v$ and say a point $q$ is a cut point of $f$ if $\nabla f(q) = 0$.

Taking now, $f(\cdot) = d(p, \cdot)$, which is semiconcave on $X \setminus \{p\}$, we define
\begin{definition}
On a neighborhood $U \subset X$ where $d(p,\cdot)$ has no cut points, a locally Lipschitz curve $\alpha: (a,b) \to U$ is called a $d_p$-gradient curve if 
\[d_p \circ \alpha(t) = t \quad \text{ and  } \quad \alpha^+(t) = \frac{\nabla d_p(\alpha(t))}{|\nabla d_p(\alpha(t))|^2}, \text{  for all }t \in (a,b).\]
\end{definition}
\noindent Here $\alpha^+(t)$ denotes the one-sided derivative 
\[\alpha^+(t) = \lim_{\epsilon \to 0^+}\frac{\log_{\gamma(t)} \gamma(t + \epsilon)}{\epsilon}, \]
where the $\log$ function is understood as follows:  for two points $p,q \in X$, let $\stackrel{\rightarrow}{pq} ~\in \Sigma_pX$ denote a direction of a minimizing geodesic in $X$ from $p$ to $q$. Define \[\log_p q := d(p,q) \cdot\vecpq \in C_pX.\] 

In \cite{PerelmanPetrunin}, using a limiting argument of directions in $\Sigma_pX$ and monotonicity estimates for $f$-gradient curves, Perelman-Petrunin show
\begin{pro} {\em (see \cite{PerelmanPetrunin})} Let $p \in X$ and $U \subset X$ a neighborhood without cut points of $d_p$. For any $q \in U$, there is a unique complete $d_p$-gradient curve starting at $q$.
\end{pro}

\noindent For a $d_p$-gradient curve $\gamma : (0,a) \to X$, which has arbitrary lower curvature bound $\geq -k^2$, define $\rho(t)$ so that
\[d\rho/\rho = dt/t \cdot |\nabla d_p(\gamma(t))|^{-2}, \quad \rho/t \to 1 \text{ as } t \to 0.\]

\begin{definition}Let $X$ be an Alexandrov space, $p \in X$. Given $v \in C_pX$, construct a complete $d_p$-gradient curve $\gamma$ with $\gamma(0) = p$ and $\gamma^{+}(0) = v/|v|$. Define the {\em gradient exponential map} $\gexp_p: C_pX \to X$ by
\[\gexp_p(v) = \gamma \circ \rho^{-1}(|v|).\]
\end{definition}
Note that the gradient-exponential map is non-expanding on $C_pX$. Gvien a direction $v \in \Sigma_pX$, we call the curve $\gamma : [0, \infty) \to X$ given by $\gamma(t) = \gexp_p(tv) $ the {\em radial curve} starting at $p$ in the direction $v$.

\subsubsection{Critical Point Theory for Alexandrov spaces} 
\label{subsection-criticalpointtheory}Essential to our arguments is Perelman's development of Morse theory for Alexandrov spaces \cite{PerelmanII, PerelmanMorse}. 

Let $(X,d) \in \Alex^n[k]$ and fix a point $p \in X$. A point $q \in X$ is called a {\it critical point of $d(p, \cdot)$} if, for any $x \neq q$, the comparison angle at $q$ satisfies $\widetilde{\angle}_{k}xqp \leq \frac{\pi}{2}$. If a point is not critical, then it is called a regular point. This metric notion of critical points of the distance function was first introduced by Grove-Shiohama and used to proved the Sphere Theorem \cite{GroveShiohama}. One of the main results of \cite{PerelmanII} is that the Isotopy Lemma of Grove-Shiohama generalizes to Alexandrov spaces as well.

In fact, we have Perelman's Fibration Theorem
\begin{thm} {\em (c.f. \cite{PerelmanII} and Proposition 2.2 of \cite{Perelman-CurvBddBelow}).}
Let $X$ be an $n$-dimensional Alexandrov space and $f: U \to \mathbb{R}^k$ an admissible function on some domain $U \subset X$. If $f$ has no critical point and is proper in $U$, then its restriction to $U$ is a locally trivial fiber bundle.  
\end{thm}

In particular, given a fixed point $p \in X$ the distance function $d(p, \cdot)$ is an admissible, proper function on $X$ (see \cite{PerelmanII, PerelmanMorse} for full definitions of an admissible function). From the Fibration Theorem, it follows 
\begin{coro} 
\label{coro-critical points}
Let $p \in X$ and suppose, for $r_2 > r_1 >0$, that  $\overline{B(p, r_2)}\setminus B(p,r_1)$ contains no critical points of $d(p, \cdot)$. Then $\overline{B(p, r_2)}\setminus B(p,r_1)$ is homeomorphic to $\partial B(p, r_1) \times [r_1, r_2]$.
\end{coro}

\subsection{Definitions of Ricci Curvature for Metric Measure Spaces}
\label{section-riccidefs}
As discussed in the Introduction, there are various definitions which generalize Ricci curvature lower bounds to the setting of metric measure spaces, and Alexandrov spaces in general. The weakest of these is the Bishop-Gromov condition of Kuwae-Shioya which we define here. 

\subsubsection{Infinitesimal Bishop-Gromov Condition}
\label{subsection-BG}
For a real number $\kappa$, set

\begin{equation}
\label{Eqn-Jacobi}
   s_{\kappa}(r)= \left\{
     \begin{array}{lr}
       \sin(\sqrt{\kappa} r) / \sqrt{\kappa} &~ \textrm{if} ~\kappa > 0\\
       r &~ \textrm{if} ~\kappa = 0\\
       \sinh(\sqrt{|\kappa|} r) / \sqrt{|\kappa}| &~ \textrm{if} ~\kappa < 0       
     \end{array}
   \right.
\end{equation}

Note that the function $s_{\kappa}$ is the solution of the Jacobi equation $\ska''(r) + \ka \ska'(r) = 0$ with initial conditions $\ska(0) = 0$ and $\ska'(0) = 1$.

For $p \in X$ and $0 <t \leq 1$, we will define a subset $W_{p,t} \subset X$ and a map $\Phi_{p,t} \to X$. The map $\Phi_{p,t}$ is called the \textit{radial expansion map}. Set $\Phi_{p,t} (p)  := p \in W_{p,t}$ and a point $x \neq p$ belongs to $W_{p,t}$ if and only if there exists a point $y \in X$ such that $x \in \overline{py}$ and $d_p(x) : d_p(y) = t : 1$, where $\overline{py}$ is a minimal geodesic from $p$ to $y$. Alexandrov spaces are necessarily nonbranching, therefore such a point $y$ as defined above is unique and we set $\Phi_{p,t}(x) :=y$. Furthermore, by the triangle inequality, one can verify that the map $\Phi_{p,t}: W_{p,t} \to X$ is locally Lipschitz. 

We can now define the infinitesimal Bishop-Gromov Condition with respect to the Hausdorff measure $\Haus^n$ on $(X,d)$. 

\begin{definition}
\label{definition-BG}
For real numbers $n \geq 1$ and $K$, we say $\Haus^n$ satisfies the {\em Infinitesimal Bishop-Gromov condition $\BG(K, n)$} if for any $p \in X$ and $t \in (0,1]$, we have 
\begin{equation}\label{BG}
d({\Phi_{p,t}}_{*} \Haus^n )(x) \leq \dfrac{t s_{\kappa}(t d_p(x))^{n-1}}{s_{\kappa}(d_p(x))^{n-1}} d\Haus^n(x).
\end{equation}
for any $x \in X$ such that $d_p(x) < \frac{\pi}{\sqrt{\kappa}}$ if $\kappa > 0$, where ${\Phi_{p,t}}_* \Haus^n$ is the push-forward of $\Haus^n$ by $\Phi_{p,t}$.
\end{definition}

\begin{rmk}
For the purposes of this paper, we will take $\kappa = 0$, thus the condition (\ref{BG}) becomes $d({\Phi_{p,t}}_{*} \Haus^n )(x) \leq t^{n-1}d\Haus^n(x)$.
\end{rmk}

Generally speaking, a complete $n$-dimensional Riemannian manifold $(M^n, g)$ with the natural volume measure satisfies the $\BG(K, n)$ condition if and only if the Ricci curvature of $M^n$ satisfies $\Ric \geq (n-1) K$, for $K \in \mathbb{R}$. In \cite{KuwaeShioya} the authors show that an $n$-dimensional Alexandrov space with curvature $\geq \kappa$ satisfies the $\BG(K, n)$ condition. Furthermore, Ohta showed that if 
$(X,d, \mu)$ satisfies the $\BG(K, n)$ condition above then the Hausdorff dimension of $X$ is bounded above by $n$.

Ultimately we want to disintegrate the Hausdorff measure over the tangent cone as a product metric over the space of directions at the point. We obtain

\begin{lem}
\label{mainlemma}
\label{integration lemma}
Let $X$ be an $n$-dimensional Alexandrov space whose Hausdorff measure $\Haus^n$ satisfies the $\BG(0,n)$ condition. For any $p \in X$ and $R>0$,
\[\Haus^n(B (p,R)) \leq  \int_{v\in \Sigma_pX} \int_0^{\min(R, \cut(v))} r^{n-1} ~dr~d\Haus^{n-1}(v).\]
\end{lem}

\begin{proof}
Following the notation set in \cite{KuwaeShioyaI},  
let $A_p \subset M$ be the union of images of minimal geodesics emanating from $p$ and define, for $r >0$,
\begin{equation}
\label{ar}
a(r) := \Haus^{n-1}\left(A_p \cap \partial(B(p,r) \right).
\end{equation}
As in \cite{KuwaeShioyaI}, for $0 < r_1 \leq r_2$ and setting $t : = r_1/r_2$, note that because of the Alexandrov convexity condition, the map $\Phi_{p,t}: A_p \cap \partial(B(p,r)) \cap W_{p,t} \to A_p \cap \partial(B(p,r_2))$ is surjective and Lipschitz continuous.  Taking $t$ very close to 1, makes the Lipschitz constant of $\Phi_{p, t}$ also close to 1. In fact, 
\[\frac{a(r_1)}{a(r_2)}\geq \left(1 - \theta(t-1 \big| \underline{k}) \right), \]
where $\underline{k}$ is a lower bound of the curvature on $X$ and the notation $\theta (x \big| y)$ indicates some function of $x \in \mathbb{R}$ depending on $y$ such that $\theta(x \big| y) \to 0$ as $x \to 0$.
From this it follows that $a(r)$ is integrable and thus, by the Co-Area formula (c.f. Theorem 4.2.1 of  \cite{LinYang}), we have 
\[
\Haus^n(B(p,R) )=  \int_0^R a(r) dr.
\] 
Also,  from the $BG(0,n)$ condition on $X$,  it follows that (summarizing \cite{KuwaeShioyaI})
\[
a(r_1) \geq \frac{r_1^{n-1}}{r_2^{n-1}} a(r_2),
\]
and thus,
\[
\frac{\log a(r_2) - \log a(r_1)}{r_2-r_1} \leq (n-1) \frac{\log r_2^{n-1} - \log r_1^{n-1}}{r_2-r_1}.
\]
Which is equivalent to 
\[
\overline{\log\circ a}~'(r) := \limsup_{h\to 0} \frac{\log(a(r+h)) - \log(a(r))}{h} \leq \frac{n-1}{r}.  
\]
It follows that $\dfrac{\overline{a}~'(r)}{a(r)} \leq \dfrac{n-1}{r}$ for any $r>0$. Now, comparing this quantity with the model space, we define the comparable quantities on  $\mathbb{R}^n$ denoting $a_0(r) = \Haus^{n-1}\left(\partial(B^0(r) \right) = \vol(\mathbb{S}^{n-1}(r))$, we naturally have $\dfrac{\overline{a_0}~'(r)}{a_0(r)} = \dfrac{a_0\,'(r)}{a_0(r)}= \dfrac{n-1}{r}$. Thus, 
\[
\frac{\overline{a}~'(r)}{a(r)} \leq \frac{a_0\,'(r)}{a_0(r)},
\]
and it follows that $\dfrac{a(r)}{a_0(r)}$ is decreasing along radial geodesics $\gamma(t) = \gexp_p(tv)$ in $X$. Note that, within the cut locus, $\partial B(p,r) = \{\gexp( rv) ~|~ v \in \Sigma_pX\}$; and therefore, it follows that $a(r) \leq r^{n-1} \Haus^{n-1}(\Sigma_pX)$. Thus, we have
\begin{equation}
\Haus^n(B (p,R)) = \int_0^R a(r) ~dr \leq \int_{v\in \Sigma_pX} \int_0^{\min(R, \cut(v))} r^{n-1} ~dr~d\Haus^{n-1}(v).
\end{equation}
\end{proof}

\section{Proof of Theorem \ref{mainthm}}
\label{section-proofs}
Throughout the rest of the paper, we assume $X \in \Alex^n[-\ka^2]$ for $\ka \in \mathbb{R}, n \geq 2$ and that assume $\partial X = \emptyset$. One of the key facts that we use is that there is a prevalence of large geodesics in the presence of large volume growth. Even more, as we will see in Lemma \ref{PerelmanMaxVolume},  given enough volume growth it is possible to place a large geodesic in an advantageous location.

\begin{lem}
\label{PerelmanMaxVolume}
For natural number $n \geq 2$, let $X \in \Alex^n[-\ka^2]$ whose Hausdorff measure satisfies the $\BG(0,n)$ condition. Given constants $\epsilon >0, C>1$, there exists a constant $\gamma(\epsilon, C, n) >0$ such that if $\Haus^n(B(p,R)) \geq (1-\gamma) \omega_nr^n$, then the following property holds:
\[ \text{ for any } a \in B(p,r), r>0, \text{ there exists } b \in X\setminus B(p,Cr) \text{  such that } d(a, \overline{pb}) \leq \epsilon r.\]
\end{lem}
\noindent Here $\omega_n$ denotes the volume of the unit ball in $\mathbb{R}^n$.
\begin{proof}
The proof is by contradiction and the argument relies on a similar computation as that found in Lemma 1.5 of \cite{Munn}, (c.f. \cite{PerelmanVolumeStability} for compact version). Since the setting is now Alexandrov spaces, we must account for the metrically singular set $S_X$. Note that if $p \in S_X$ then $\vol(S_X) < \vol(\mathbb{S}^{n-1}) = \pi$ and thus the bounds obtained on the $\Haus^n$-measure of geodesics balls is strictly smaller and thus gives only a sharper estimate. To make this clear, and for the ease of the reader, we include the necessary computation here. 

Suppose $p \in S_X$, then $\Haus^{n-1}(\Sigma_pX) < \vol(\mathbb{S}^{n-1})$ and set \\$\Gamma = \{v \in \Sigma_pX ~|~ d(a, \gexp_p(tv)) \leq \epsilon R, t \geq 0\}$. Assume (by way of contradiction) that for all $v \in \Gamma$,
\[
\sup \{t >0~|~ \gexp_p(tv)  \text{ is minimizing}\} < Cr.
\]
Taking $\Cbar > C >1$, we have by way of Lemma \ref{integration lemma} 
\begin{eqnarray*}
\Haus(B(p,\Cbar r)) &\leq&\int_{v\in\Sigma_pX} \int_{0}^{\min(\Cbar r, \cut(v))} t^{n-1} ~dt~\Haus^{n-1}(dv)  \\
				&\leq& \int_{v\in\Gamma} \int_{0}^{Cr} t^{n-1} ~dt ~\Haus^{n-1}(dv)+  \int_{v\in\Sigma_p X\setminus \Gamma}\int_{0}^{\Cbar r} t^{n-1} ~dt ~\Haus^{n-1}(dv)\\
				&\leq& \Haus^{n-1}(\Gamma) \int_{0}^{Cr} t^{n-1} ~dt + \Haus^{n-1}(\Sigma_pX \setminus \Gamma) \int_{0}^{\Cbar r} t^{n-1} ~dt\\
				&<& -\Haus^{n-1}(\Gamma) \int_{Cr}^{\Cbar r} t^{n-1} ~dt + \Haus^n(B^0(\Cbar r)).
\end{eqnarray*}
In the last inequality we used the fact that $\Haus^{n-1}(\Sigma_pX) < \vol(\mathbb{S}^{n-1})$. Thus, 
\[\displaystyle{(1-\gamma) \Haus^n(B^0(\Cbar r)) < -\Haus^{n-1}(\Gamma) \int_{Cr}^{\Cbar r} t^{n-1}~dt + \Haus^n(B^0(\Cbar r))}.\] Simplifying, we get a lower bound on the size of $\Gamma$:
\[
\Haus^{n-1}(\Gamma) < \frac{\gamma \Haus^n(B^0(\Cbar r))}{\int_{Cr}^{\Cbar r} t^{n-1}~dt} = \frac{\gamma~ n~ \omega_n \Cbar^n}{\Cbar^n - C^n}.
\]
This upper bound on the size of $\Gamma$ then allows us to, in turn, bound the $\Haus^n$-measure of $B(a, \epsilon r)$:
\begin{equation}
\label{upperbound}
\Haus^n(B(a, \epsilon r)) 
	\leq \Haus^{n-1}(\Gamma) \int_{0}^{Cr} t^{n-1} ~dt \leq \dfrac{\Cbar^n \gamma \omega_n (Cr)^n}{\Cbar^n - C^n}.
\end{equation}

By assumption $\Haus^n$ satisfies the $BG(0,n)$ condition so, by Corollary 3.4 in \cite{KuwaeShioyaI}, we can obtain a lower bound for the $\Haus^n$-measure of $B(a, \epsilon r)$. Namely, 
\begin{eqnarray}
\Haus^n(B(a, \epsilon r )) &\geq& \Haus^n(B(a, r + \Cbar r)) \frac{\epsilon^n}{(1+\Cbar)^n}\\
					\label{lowerbound}&\geq& \Haus^n(B(p, \Cbar r)) \frac{\epsilon^n}{(1+\Cbar)^n}\geq (1-\gamma) \omega_n (\Cbar r)^n \frac{\epsilon^n}{(1+ \Cbar)^n}.
\end{eqnarray}
Combining (\ref{upperbound}) and (\ref{lowerbound}), and solving for $\gamma$, we obtain the bound
\[\gamma \geq \left[ 1 + \frac{(1+\Cbar^n)}{\epsilon^n} \frac{C^n}{\Cbar^n-C^n}\right]^{-1}.\]
Since $\Cbar$ was taken arbitrary we can let $\Cbar \to \infty$ and get a lower bound for $\gamma$ depending only on $\epsilon, C$ and $n$. Setting $\gamma(\epsilon, C, n) < \left[ 1 + \frac{C^n}{\epsilon^n}\right]^{-1}$ we obtain a contradiction and the Lemma is proved.

\end{proof}
Adapting techniques familiar from classical Riemannian geometry, we prove
\begin{thm}
For an integer $n \geq 2$, let $(X,d) \in \Alex^n[-\ka^2]$ be a complete, noncompact Alexandrov space whose Hausdorff measure satisfies the $BG(0,n)$ condition. There exists an $\epsilon(n,\ka) >0$ such that if, for $p \in X$
\[ \Haus^n(B(p,r)) \geq  (1-\epsilon)\omega_n r^n, \quad \text{ for all } r >0,\]
then $(X,d)$ has finite topological type. 
\end{thm}

\begin{proof}
By Corollary \ref{coro-critical points}, it suffices to show that the distance function $d(p, \cdot)$ has no critical points outside some large ball. Suppose by contradiction that we can find a sequence of critical points $\{x_i\}_{i=1}^{\infty}$ such that $d(p, x_i) \to \infty$ as $i \to \infty$. Denote $R_i = d(x_i, p)$. 

Given $n$ and $\kappa$ as in the statement of the Theorem, choose some $\epsilon(n,\kappa) < \left(\frac{\ln(2)}{8\kappa} \right)^{1-1/n}$. Following the proof of Lemma \ref{PerelmanMaxVolume}, since $\alpha_X \geq \alpha(n, \kappa):= 1-\left(1+ \frac{2^n}{\epsilon^n} \right)^{-1}$, it follows that for each $x_i$, there exists some $b_i \in X \setminus B(p, 2R_i)$ such that the minimal geodesic $\overline{pb_i}$ connecting $p$ to $b_i$ satisfies $d(x_i, \overline{b_i}) \leq \epsilon R_i$. Let $q_i \in \overline{pb_i}$ such that $d(p, q_i) = 2R_i$. 

For each $k$, let $\beta_i$ be an arbitrary minimizing path connecting $x_i$ to $q_i$. Since $x_i$ is a critical point of $d(p, \cdot)$ there exists a minimizing path $\alpha_i$ connecting $x_i$ to $p$ such that $\cos \angle(\beta'(0), \alpha'(0)) \leq \pi/2$. Consider now the geodesic triangle  $\Delta(\alpha_i, \beta_i, \overline{pq_i})$ formed by these paths joined with $\overline{pq_i}$. 

Note that 
\[h(x_i) := d(x_i,\overline{pq_i}) \leq \epsilon R_i,\]
and 
\[s(x_i) := \min \{d(p, x_i), d(x_i, q_i) \}= R_i.\]
Thus, by the excess estimate for Alexandrov spaces (see (\ref{excess}) in the Appendix) applied to the triangle $\Delta px_iq_i$, we have
\begin{eqnarray}
e_{p,q_i}(x_i) &\leq& 8 \left( \frac{h(x_i)^n}{s(x_i)}\right)^{1/(n-1)}\\
		& \leq & 8  \left( \frac{(\epsilon R_i)^n}{R_i}\right)^{1/(n-1)}\\
		& \leq &  8 \epsilon^{n/(n-1)} < \frac{1}{\kappa}\ln(2). \label{excessupperbound}
\end{eqnarray}

Similarly, we can apply the Toponogov Comparison theorem for Alexandrov spaces to  $\Delta(\alpha_i, \beta_i, \overline{pq_i})$ to bound the excess from below. Since $\angle(px_iq_i) \leq \pi/2$, the hyperbolic law of cosines ensures that 
\begin{eqnarray*}
\cosh \ka d(p,q_i) &=& \cosh \ka d(p, x_i) \cosh \ka d(x_i, q_i) - \sinh \ka d(p, x_i) \sinh \ka d(x_i, q_i) \cos \angle(px_iq_i)\\
				&\leq& \cosh \ka d(p, x_i) \cosh \ka d(x_i, q_i).
\end{eqnarray*}
Thus, it follows
\[e^{\ka d(p, q_i)} \leq \cosh(\ka R_i) e^{\ka d(x_i,q_i) -\ka R_i},\]
which further simplifies to give
\[d(x_i, q_i) \geq d(p, q_i) - \frac{1}{\ka} \ln\left(\frac{1-e^{-2\ka R_i}}{2} \right) .\]
Therefore, since $d(p, q_i) = 2R_i, d(p, x_i) = R_i$,
\begin{eqnarray}
e_{p,q_i}(x_i) &=& d(p,x_i)  + d(q_i, x_i) -d(p, q_i)\\
			&\geq &  - \frac{1}{\ka} \ln\left(\frac{1-e^{-2\ka R_i}}{2} \right)=  \frac{1}{\ka} \ln\left(\frac{2}{1-e^{-2\ka R_i}} \right).
\end{eqnarray}
Since $\frac{1}{\ka}\ln\left(\frac{2}{1-e^{-2\ka R_i}} \right) \searrow\frac{1}{\ka} \ln(2)$ as $R_i \to \infty$, we arrive at a bound for $e_{p,q_i}(x_i)$ contradicting (\ref{excessupperbound}) for some $R_i$ large. Thus, there are no critical points outside the ball $B(p, R_i)$ and the theorem is proved. 
\end{proof}

\begin{rmk}
In \cite{KuwaeShioyaII} the authors prove a Laplacian comparison theorem and topogical splitting theorem for weighted Alexandrov spaces, i.e. Alexandrov spaces equipped with a measure $\mu$ (not necessarily the Hausdorff measure) which satisfies the $\BG(0,n)$ condition. One would anticipate that since the Laplacian comparison plays an integral role in the proof of the excess estimate (as we see in the Appendix) and thus the proof above, a comparable theorem should hold for a metric measure space $(X,d,\mu)$ where $(X,d) \in \Alex^n[-\ka^2]$ and $\mu$ which satisfies the $\BG(0,n)$ condition. We have not worked out the explicit details of this argument and but leave it for the interested reader. 
\end{rmk}

\section{Appendix}
\label{section-appendix}
In this section we sketch proof of the excess estimate in Alexandrov spaces and describe the necessary ingredients. The original excess estimate was proven by Abresch-Gromoll in \cite{AbreschGromoll1990} for Riemannian manifolds with Ricci curvature bounded below by $(n-1)H$, for $H \in \mathbb{R}$. We adapt their argument to generalize the excess estimate to $n$-dimensional Alexandrov spaces whose Hausdorff measure satisfies a condition giving the space the property of non-negative Ricci curvature (see Section \ref{section-introduction} for discussion of various notions of Ricci curvature lower bounds in Alexandrov space and metric measure spaces in general). 

Throughout, let $(X,d)$ denote an $n$-dimensional Alexandrov space with curvature $\geq -\ka^2$ for some $\ka \in \mathbb{R}$. Central to our proof of the excess estimate is an analysis of harmonic functions on $X$. The theory of (sub/super) harmonic functions on Alexandrov spaces has been developed by Kuwae-Machigashira-Shioya, Petrunin, Kuwae-Shioya and others \cite{KMS, PetruninHarmonic, KuwaeShioyaI, KuwaeShioyaII} in reference to an underlying $n$-dimensional Hausdorff measure $\Haus^n$ on the $(X,d)$. As initiated in \cite{KuwaeShioyaII}, one could also consider Alexandrov spaces equipped with a positive Radon measure $\mu$ satisfying some generalized weak-Ricci curvature lower bound condition; thus studying metric measure spaces $(X,d,\mu)$ where $(X,d)$ is an Alexandrov space of curvature $\geq -\ka^2 > - \infty$ and $\mu$ a Radon measure satisfying $BG(0,n)$. We state the necessary principles below for $\mu$ a Radon measure keeping in mind that ultimately we will take $\mu$ to be $ \Haus^n$. 

\subsection{Laplacian Comparison and Weak Maximum Principle on Alexandrov spaces}
\label{subsection-laplacian}
Both the Laplacian comparison and the weak maximum principle are very a important tools when studying Riemannian manifolds with a lower Ricci curvature bound. Following \cite{KuwaeShioyaII}, here we collect the tools and terminology we need to describe the  Laplacian comparison on weighted Alexandrov spaces $(X,d,\mu)$ as well as the weak maximum principle. We refer the reader to Section \ref{subsection-alexandrov} for notation.

\subsubsection{Maximum Principle for $\mu$-subharmonic Functions}
\label{subsection-maximum}
The original proof of the excess estimate for Riemannian manifolds \cite{AbreschGromoll1990} relies on the maximum principle for subharmonic functions. However, since the Riemannian metric of $(X,d)$ has low regularity we cannot simply apply the usual maximum principle in this setting. Instead, we follow work of Kuwae \cite{Kuwae} and employ theory of $\mu$-subharmonic functions and the maximum principle in this setting. Here we describe in more detail these concepts.

\subsubsection{Harmonic Functions} Let $\Omega \subset X$ be a relatively compact open, connected domain and denote $\Lip(\Omega)$ the set of Lipschitz functions on $\Omega$. Using the natural $C^0$-Riemannian structure on the set of nonsingular points of $X$ we can define $L^2$ and  $W^{1,2}$ functions and their norms. The space of Sobolev functions $W^{1,2}(\Omega, \mu)$ is taken to be the closure of $\Lip(\Omega)$ in the norm:
\[|| f ||^2_{W^{1,2}(\Omega, \mu)} := \int_{\Omega} (f^2 + |\nabla f|^2) ~d\mu,  \]
where $\nabla f $ is the maximal rate of growth of $f$ at a point. Let $W^{1,2}_0(\Omega, \mu)$ denote the $W^{1,2}$-closure of the set of $W^{1,2}$ functions of compact support in $\Omega$ and define a canonical {\em Dirichlet form} by 
\[\mathcal{E}^{\mu}(f,g) :=  \int_{\Omega} \langle \nabla f, \nabla g \rangle ~d\mu, \quad \text{ for all } f, g \in \Lip_0(\Omega).\]
Note that here $\langle \cdot , \cdot \rangle$ denotes the inner product induced from the Riemannian structure on $X$ outside the singular set and thus is defined $\mu$-a.e. since locally Lipschitz functions on $\Omega$ are differentiable $\mu$-a.e. on $\Omega$. 

Now we can define
\begin{definition}
A function $f \in W^{1,2}_{0, \loc}(\Omega; \mu)$ is said to be $\mu$-subharmonic provided 
\[
\int_{\Omega} \left< \nabla f, \nabla \phi \right> d\mu \leq 0
\]
for any nonnegative smooth function $\phi \in \Lip_0(\Omega)$. 
\end{definition}

For which we have the following maximum principe of Kuwae-Shioya \cite{KuwaeShioyaII}

\begin{thm}
\label{thm-maxprinciple}
Let $f \in W^{1,2}_{0, \loc}(\Omega; \mu)$ be a continuous $\mu$-subharmonic function and suppose $f$ attains its maximum on $\Omega$. Then, $f$ is constant on $\Omega$. 
\end{thm} 

\subsubsection{Laplacian comparison for the distance function}

Now fix a point $p \in X$ and denote $d(p,\cdot) :X \to \mathbb{R}$ the distance function from $p$. This map is differentiable outside the singular set and away from Cut$_p \cup  \{p\}$ and the gradient field $\nabla d(p, \cdot)$ is continuous at all differentiable points. Because of the lack of regularity of the Riemannian metric on an arbitrary Alexandrov space,  the standard proof of the Laplacian comparison theorem for Riemannian manifolds with lower Ricci curvature bounds does not work in our setting. However, up to a set of measure zero we have by \cite{KuwaeShioyaI}

\begin{thm}
\label{laplaciancomparison}
For a positive Radon measure $\mu$ on $(X,d)$ with full support and which satisfies the $BG(0,n)$ condition on $\Omega$, we have
\[\int_X \langle \nabla d_p, \nabla \phi \rangle ~d\mu \geq \int_X \left(-\frac{n-1}{d_p} \right) ~\phi ~d\mu,\]
for any nonnegative function $\phi \in \Lip_0(\Omega \setminus \{p\}).$
\end{thm}

Note, the Laplacian Comparison theorem above can be stated more generally for Radon measures satisfying $BG(k,n)$ for $k,n \in \mathbb{R}$. We state only the case $k=0$ here because it applies directly to our arguments in Section \ref{section-proofs} and to simplify the exposition.

Equipped with the Laplacian comparison theorem and the maximum principle for sub-harmonic functions, we are now have the ingredients we need to prove the excess estimate. Since the proof follows the original argument of Abresch-Gromoll \cite{AbreschGromoll1990} we include only a brief sketch here (for the case $BG(0,n)$) just to convey the main ideas.

\subsection{Excess Estimate for Alexandrov Spaces}
For any metric space $(X,d)$ and points $p,q \in X$, the excess function is given by 
\[e_{p,q}(x) = d(x,p) + d(x,q) - d(p,q), \quad \text{ for } x \in X.\]
Note that $e_{p,q}$ is necessarily a Lipschitz function and has Lipschitz constant at most 2. Setting $h(x) := d(x, \overline{pq})$, the height of the triangle $\Delta pxq$, it follows immediately from the triangle inequality that $0 \leq e_{p,q}(x) \leq 2h(x)$, for $x \in X$.

The strength of the excess estimate is that it guarantees a stronger upper bound in the presence of a lower Ricci curvature bound. And, as we will show, the estimate is particularly useful to long thin long thin triangles. In this subsection, we use the tools outlined above to sketch the proof of the excess estimate for Alexandrov spaces with $\BG(0,n)$. The argument follows the original proof of Abresch-Gromoll.

As motivated by the argument in in \cite{AbreschGromoll1990}, define the function $\varphi_{n, K}: (0, l] \times \mathbb{R} \to [0, \infty)$ by
\begin{equation}
\label{phi}
\varphi_{n,\kappa}(r, l) = \iint_{r \leq t\leq \tau \leq l} \left( \frac{s_{\kappa}(\tau)}{s_{\kappa}(t)}\right)^{n-1} ~d\tau dt,
\end{equation}
where the functions $s_{\kappa}(\cdot)$ are the Jacobi equations given in (\ref{Eqn-Jacobi}) and when $K >0$ we assume $Kl \leq \pi^2$.

This $\varphi_{n,K}(\cdot, \cdot)$ is used in proving the general form of the excess estimate of Abresch-Gromoll for Riemannian manifolds with $\Ric \geq (n-1)K$. Since we focus here only on the case when $K=0$, (\ref{phi}) becomes
\[
 \varphi_{n, 0} (r, l) = \iint_{r \leq t \leq \tau \leq l} \left(\frac{\tau}{t}\right)^{n-1} ~d\tau dt = \frac{1}{2n} \left(r^2 -\frac{n}{n-2}l^2 + \frac{2}{n-2} l^n r^{-(n-2)} \right).
\]
Fixing $l>0$ and considering $\varphi$ as a function on $r \in (0,l]$, the relevant properties of $\varphi$ are that 
\begin{enumerate}[(a)]
\item $\varphi(\cdot,l) >0$ and is decreasing on $(0, l)$
\item $\varphi(l,l) = 0$
\item $\Delta \varphi(\cdot,l) = 1$
\end{enumerate}
We show

\begin{thm}
For an integer $n \geq 2$ and $\ka \in \mathbb{R}$, let $(X,d) \in \Alex^n[-\ka^2]$ and suppose the Hausdorff measure $\Haus^n$ on $X$ satisfies the $\BG(0,n)$ condition of Definition \ref{definition-BG}.
Fix $p, q \in X$ and let $x\in X$ so that $h(x) \leq \frac{1}{2}\min\{d(p,x), d(q,x)\}$. Then,
\begin{equation}
\label{excess}
e_{p,q}(x) \leq 8 \left(\frac{h(x)^n}{s(x)} \right)^{\frac{1}{n-1}},
\end{equation}
where $s(x) = \min\{d(p,x), d(q,x)\}$.
\end{thm}
In particular, note that as either $h(x) \searrow 0$ or $s(x) \nearrow \infty$, it follows that $e_{p,q}(x) \searrow 0$. 

\begin{proof}
Note that by Theorem \ref{laplaciancomparison}, it follows that 
\[
\Delta e_{p,q}(x) \leq \frac{n-1}{d(p,x)} + \frac{n-1}{d(q,x)} \leq \frac{2(n-1)}{s(x)}, ~~\text{ in the distributional sense on }X.
\]
Let $x \in X$. For some $\epsilon >0$, let $\tilde{l} = h(x) + \epsilon$ and  consider the function $G: B(x, \tilde{l}) \to [0, \infty)$ defined by
\begin{eqnarray}
G(\cdot) &=&  \frac{2(n-1)}{s(x)} \varphi_{n,0}(d(x, \cdot), \tilde{l})\\
		&=& \frac{2(n-1)}{s(x)}  \frac{1}{2n} \left(d(x, \cdot)^2 -\frac{n}{n-2}\tilde{l}^2 + \frac{2}{n-2} \tilde{l}^n d(x, \cdot)^{-(n-2)} \right).
\end{eqnarray}

Note that by the properties of $\varphi$, it follows that 
\begin{enumerate}[(a)]
\item $G >0$ on $B(x, \tilde{l})$
\item $G = 0$ on $\partial B(p, \tilde{l})$
\item $\Delta G = \frac{2(n-1)}{s(x)}$ in the distributional sense on $X$
\end{enumerate}

We will show that 
\begin{equation}
\label{excessbound}
e_{p,q}(x) \leq 2c + G(c),~~ \text{ for all } c \in (0, h(x)).
\end{equation}
Suppose not. Then there exists some $\overline{c} \in (0, h(x))$ such that $e_{p,q}(x) > 2\overline{c} + G(\overline{c})$. Note that, since the Lipschitz constant of $e_{p,q}$ is at most 2, it follows that for any $y \in \partial B(x, \overline{c})$ that 
\[
e_{p,q}(x) - e_{p,q}(y) \leq 2 d(x,y).
\]
Therefore, for any $y \in \partial B(x, \overline{c})$
\begin{eqnarray*}
e_{p,q}(y) &\geq& e_{p,q}(x) - 2d(x,y)\\
		&=& e_{p,q}(x) -2\overline{c}, \quad\text{ since } y \in  \partial B(x, \overline{c})\\
		&>& 2\overline{c} + G(\overline{c}) -2\overline{c} =  G(\overline{c}).
\end{eqnarray*}
Thus, 
\[ \left(G - e_{p,q}\right)\big|_{\partial B(x, \overline{c})} <0.
\]
Similarly, by property (b) of $G$ above and since $e_{p,q} > 0$ away from the geodesic $\overline{pq}$, we have that 
\[
\left(G - e_{p,q} \right) \big|_{\partial B(x, h(x) + \epsilon)} <0.
\]
Lastly, we note that for $z\in \overline{pq}$ such that $d(x,z) = h(x)$, by property (a) of G and since $e_{p,q}(z) = 0$, it follows that 
\[ G(z) - e_{p,q}(z) <0.
\] 
Therefore, the point $z$ is a local maximum of the function $G(\cdot) - e_{p,q}(\cdot)$. However, by the linearity of the Laplacian, and property (c) of $G$ above we have
\[\Delta \left(G-e_{p,q} \right) =  \Delta G - \Delta e_{p,q} \geq \frac{2(n-1)}{s(x)} - \frac{2(n-1)}{s(x)} =0 .\]
That is to say, $G(\cdot) - e_{p,q}(\cdot)$ is an $\Haus^n$-subharmonic function on $B(x, h(x) + \epsilon) \setminus B(x, \overline{c})$. Thus, Theorem \ref{thm-maxprinciple}, taking $\Omega = B(x, h(x) + \epsilon) \setminus B(x, \overline{c})$, implies that  $G(\cdot) - e_{p,q}(\cdot)$ must be constant which is a contradiction.

Taking $\epsilon \searrow 0$ we arrive at our desired result and (\ref{excessbound}). Further simplifying that upper bound, taking $c = \frac{2h^n(x)}{s(x)}$ and recalling that $h(x) \leq \frac{1}{2} s(x)$, we obtain
\[
e_{p,q}(x) \leq 2c + G(c) \leq 8\left( \frac{h(x)^n}{s(x)}\right)^{\frac{1}{n-1}}.
\]
\end{proof}
\begin{rmk}
See also Proposition 6.2 of \cite{ChCoWarped} (c.f. Theorem 9.1 of \cite{CheegerDegeneration}) which gives a slight generalization (not assuming a zero of $e_{p,q}$)  of the Abresch-Gromoll excess estimate and has significant applications which are useful in studying Gromov-Hausdorff limits. In \cite{GigliMosconi} the authors also generalize this Cheeger-Colding version to infinitesimally Hilbertian $\CD(K,n)$ spaces. We expect a similar argument would work in our setting for Alexandrov spaces as well.
\end{rmk}

\end{document}